\documentclass[reqno]{amsart}

\usepackage[absolute,overlay]{textpos}
\usepackage{eso-pic}
\usepackage[a4paper,margin=1.25in]{geometry}
\usepackage{mathrsfs}
\usepackage{amsfonts,  amsmath}
\usepackage{graphicx}
\usepackage{enumerate,  enumitem,  amscd}
\usepackage{amsthm,  amssymb,  epsfig,  hyperref,  amsmath}  
\usepackage{mathtools}
\usepackage{tikz}
\usepackage{tikz-cd}
\usetikzlibrary{positioning, arrows,automata,arrows.meta}
\usepackage{quiver}
\usepackage{subcaption}
 
\tikzset{
  vx/.style  = {circle, draw=black, fill=blue!20, line width=0.6pt,
                minimum size=8pt, inner sep=0pt},
  ed/.style  = {draw=black!60, line width=0.5pt},
  lbl/.style = {font=\tiny},
  ttl/.style = {font=\small}
}

\newcommand{\pres}[3]{\textnormal{#1} \langle #2 \mid #3 \rangle}

\newcommand{\ZZ}{\mathcal{Z}}

\newcommand{\Z}{\mathbb{Z}}

\newcommand{\cF}{\mathcal{F}}

\newcommand{\B}{\mathbf{B}}

\DeclareMathOperator{\SL}{SL}

\DeclareMathOperator{\Mod}{Mod}
\DeclareMathOperator{\Out}{Out}

\newtheorem{theorem}{Theorem} 
\newtheorem*{theorem*}{Theorem} 
\numberwithin{theorem}{section}
\newtheorem{lemma}[theorem]{Lemma}     

\newtheorem*{corollary*}{Corollary}
\newtheorem{proposition}[theorem]{Proposition}
\newtheorem*{proposition*}{Proposition}

\newtheorem*{question*}{Question}
\numberwithin{question}{section}

\theoremstyle{definition}

\numberwithin{example}{section}

\newtheorem*{remark*}{Remark}

\begin{document}

\title[Braid groups are not profinitely rigid]{Braid groups are not profinitely rigid}
\author{Carl-Fredrik Nyberg-Brodda}
\address{June E Huh Center for Mathematical Challenges, Korea Institute for Advanced Study (KIAS), Seoul 02455, Korea}
\email{cfnb@kias.re.kr}

\thanks{The author is supported by KIAS Individual Grant HP094701 at Korea Institute for Advanced Study, and by the Mid-Career Researcher Program (RS-2023-00278510) through the National Research Foundation of the Government of Korea.}

\date{\today}

\subjclass[2020]{}

\begin{abstract} 
We show that the $n$-strand braid group $\B_n$ is not profinitely rigid when $n \geq 4$. We do this by exhibiting a family of pairwise non-isomorphic, finitely presented residually finite groups with the same profinite completion as $\B_n$. This answers a well-known problem negatively, and gives the first known examples of non-profinitely rigid mapping class groups. 
\end{abstract}

\maketitle

\noindent A finitely generated, residually finite group $G$ is said to be (absolutely) profinitely rigid if it is determined amongst all finitely generated, residually finite groups by its set of finite quotients. In recent years, a number of results have demonstrated profinite rigidity arising naturally in many settings, with the most spectacular being examples of profinitely rigid hyperbolic $3$-manifold groups \cite{BMRS}. Many other classes, including free-by-cyclic groups, have since been studied from the point of view of profinite rigidity \cite{Bridson2017, Hughes2025}. A natural question, explicitly noted as an open problem by Andrew, Hensel, Hughes \& Wade \cite[pp.\ 9--10]{Andrew2026}, concerns the profinite rigidity of mapping class groups, including the famous \textit{braid groups} $\B_n$, being the mapping class groups of the $n$-punctured disk. In relation to this, finite quotients of braid groups have seen a good deal of study in recent years, see e.g.\ \cite{Scherich2023, Kolay2023}. In this note, we answer the question of the profinite rigidity of braid groups negatively. Specifically, we show:

\begin{theorem*}
For $n \geq 4$, the $n$-strand braid group $\B_n$ is not profinitely rigid. 
\end{theorem*}

This gives the first known examples of non-profinitely rigid mapping class groups. Our methods do not seem to give any insight into the profinite rigidity of other mapping class groups, and e.g.\ the profinite rigidity of $\Out(F_n)$ and $\Mod(\Sigma_g)$ for genus $g$ surfaces remain open problems.

\section{}

\noindent In this section, we briefly recall the essential definitions which will be used throughout. Let $G$ be a finitely generated residually finite group, and denote by $\cF(G)$ the set of all finite quotients of $G$, given the structure of an inverse system in the natural way. The profinite completion $\widehat{G}$ of $G$ is the inverse limit of this system. The group $G$ is said to be \textit{profinitely rigid} if $\widehat{G} \cong \widehat{H}$ implies $G \cong H$ for any other finitely generated, residually finite group $H$. The condition $\widehat{G}_1 \cong \widehat{G}_2$ is equivalent to $\cF(G_1) = \cF(G_2)$, i.e.\ that $G_1$ and $G_2$ have the same finite quotients \cite{Dixon1982, Nikolov2007}. Thus profinite rigidity asks whether the set of finite quotients of a group distinguishes it among all other finitely generated, residually finite groups. In recent years, many results for proving profinite rigidity have appeared via hyperbolic geometry in work by Bridson, McReynolds, Reid \& Spitler \cite{BMRS, BMRS21}.

On the other hand, one of the oldest results for proving \textit{non-}profinite rigidity is due to Baumslag \cite[p. 250]{Baumslag1974}, who gave examples of non-isomorphic metacyclic groups with the same finite quotients. The key tool in his short argument is the following elegant proposition, the proof of which uses some results from the theory of group varieties. 

\begin{proposition}[Baumslag, 1974]\label{Prop:Baumslag}
If $A, B, C, D$ are finitely generated groups such that $\widehat{B} \cong \widehat{D}$, and $A \times B \cong C \times D$, then $\widehat{A} \cong \widehat{C}$, i.e.\ $A$ and $C$ have the same finite quotients. 
\end{proposition}

A direct corollary of his result, also used in the construction of non-profinitely rigid metacyclic groups in \cite{Baumslag1974}, is that if $A \times \Z \cong C \times \Z$, then $A$ and $C$ have the same set of finite quotients. This is precisely the tool we will use in the proof of our Theorem. We now turn to this proof.

\clearpage

\section{}

\noindent \textit{Proof of Theorem.} Fix $n \geq 4$, and let $\B_n$ be the $n$-strand braid group, with standard generators $\sigma_1, \dots, \sigma_{n-1}$. The center $\ZZ(\B_n)$ of $\B_n$ is cyclic, and is generated by the full twist element $\Delta^2 := (\sigma_1 \cdots \sigma_{n-1})^n$. Let $N = n(n-1)$, which is the image of $\Delta^2$ under the abelianization map $\B_n \xrightarrow{\operatorname{ab}} \Z$. Let $Q_n = \B_n / \langle \Delta^2 \rangle$ be the quotient of $\B_n$ by its center. For any unit $r \in (\Z / N \Z)^\times$, we then define the groups $G_{n,r}$ by 
\begin{equation}\label{Eq:gnr}
G_{n,r}=\left\langle x_1, \dots, x_{n-1}, z \ \middle|\
\begin{aligned}
&[x_i, x_j] = 1 \quad (|i-j|>1), \qquad x_i x_{i+1} x_i = x_{i+1} x_i x_{i+1}, \\
&(x_1 \cdots x_{n-1})^n=z^r, \qquad \qquad\quad\qquad  z \text{ central}\\
\end{aligned}
\right\rangle .
\end{equation}
This is a fibre product of the two homomorphisms of $Q_n \xrightarrow{p_1} \Z / N\Z$ (passing via the abelianization map and multiplying by $r$) and $\Z \to \Z / N\Z$ (reduction modulo $N$). We note that $\ZZ(G_{n,r}) = \langle z \rangle$, and that $G_{n,r} / \ZZ(G_{n,r}) \cong Q_n$. Furthermore, it is easy to verify that $G_{n,1} \cong \B_n$. Indeed, if $r=1$, then the second row of relations in \eqref{Eq:gnr} simply asserts that the element $\Delta^2 = (x_1 \cdots x_{n-1})^n = z$ is central and that we can remove this redundant generator, leaving the usual presentation of $\B_n$. We will show that the groups $G_{n,r}$ all have the same profinite completion as $\B_n$. 

Since $G_{n,1} \cong \B_n$, the remainder of the proof now reduces to proving two lemmas: 

\begin{lemma}\label{Lem:not-isomorphic}
If $r, s \in (\Z / N \Z)^\times$, then $G_{n,r} \cong G_{n,s}$ if and only if $r \equiv \pm s \pmod{N}$. 
\end{lemma}

\begin{lemma}\label{Lem:same-quotients}
For every $r, s \in (\Z / N \Z)^\times$, the groups $G_{n,r}$ and $G_{n,s}$ have the same finite quotients.
\end{lemma}

We first prove Lemma~\ref{Lem:not-isomorphic}. The reverse direction is easy; indeed, if $r \equiv \pm s \pmod{N}$, then write $r \equiv \varepsilon s + Nm$ for some $\varepsilon = \pm 1$ and $m \in \Z$, and define $\theta \colon G_{n,r} \to G_{n,s}$ by $\theta(x_i) = y_i^\varepsilon z_s^m$ and $\theta(z_r) = z_s$. Then $\theta$ is easily checked to be a homomorphism, and it is clearly both surjective and injective, since it induces either the identity or the mirror automorphism on $Q_n$, and it maps the center $\langle z_r \rangle$ isomorphically onto $\langle z_s \rangle$. Hence the reverse implication holds. 

We now prove the forward implication. This can likely be proved using combinatorial group theory via the presentation \eqref{Eq:gnr}; we instead do it by appealing to a topological result. Bell \& Margalit \cite[Main Theorem 3]{Bell2006} proved that any injective endomorphism $\alpha \colon Q_n \to Q_n$ is induced by a homeomorphism $h$ of the $n$-punctured disk; concretely, if $T_a$ denotes the image in $Q_n$ of a half-twist about an arc $a$, then $\alpha(T_a) = T_{h(a)}^{\pm 1}$. Let $z_r$ and $z_s$, respectively, denote the central generators of $G_{n,r}$ and $G_{n,s}$ (the element $z$ in \eqref{Eq:gnr}), let $x_1, \dots, x_{n-1}$ resp.\ $y_1, \dots, y_{n-1}$ be the other generators, and let $\pi_r \colon G_{n,r} \to Q_n$ resp.\ $\pi_s \colon G_{n,s} \to Q_n$ denote the projections. 

Assume then that $\theta \colon G_{n,r} \to G_{n,s}$ is an isomorphism. Then $\theta$ maps the characteristic subgroup $\ZZ(G_{n,r}) = \langle z_r \rangle$ to the center of $G_{n,s}$, so $\theta(z_r) = z_s^k$ for some $k = \pm 1$. It also induces an automorphism $\alpha \colon Q_n \to Q_n$ such that $\pi_s \circ \theta = \alpha \circ \pi_r$. Let $(a_1, \dots, a_{n-1})$ be the standard chain of arcs in the punctured disk, so that $T_{a_i} = \pi_r(x_i) = \pi_s(y_i)$. By the result of Bell \& Margalit cited above, there is a homeomorphism $h$ of the punctured disk and a sign $\varepsilon = \pm 1$ such that $\alpha(T_{a_i}) = T_{h(a_i)}^{\varepsilon}$ for every $i$. If $h$ preserves orientation, let $q \in Q_n$ be its mapping class. Since $q T_a q^{-1} = T_{h(a)}$, we get $\alpha(T_{a_i}) = q T_{a_i} q^{-1}$. On the other hand, if $h$ reverses orientation, then choose a reflection $\rho$ which preserves each standard arc $a_i$ setwise; then $h \rho$ preserves orientation and $(h\rho)(a_i) = h(a_9)$, so taking $q \in Q_n$ to be the mapping class of $h \rho$ we again get $T_{h(a_i)} = q T_{a_i} q^{-1}$ and so $\alpha(T_{a_i}) = q T_{a_i}^{-1} q^{-1}$. Thus in either case, there is a $q \in Q_n$ and $\varepsilon = \pm 1$ such that $\alpha(T_{a_i}) = q T_{a_i}^{\varepsilon} q^{-1}$ for all $i$. Choose a lift $\widetilde{q} \in G_{n,s}$ of $q \in Q_n$ with $\pi_s(\widetilde{q})=q$, and conjugate $\theta$ by $\widetilde{q}^{-1}$. This does not change the image of the central element of $z_r$, and the induced automorphism on $Q_n$ will now send $T_{a_i}$ to $T_{a_i}^\varepsilon$. Hence $\pi_s(\theta(x_i)) = T_{a_i}^\varepsilon = \pi_s(y_i^\varepsilon)$. Hence for every $i$, we find that
\begin{equation}\label{Eq:theta_xi}
\theta(x_i) = y_i^\varepsilon z_s^{m_i} \quad (1 \leq i \leq n-1)
\end{equation}
for some $m_i \in \Z$, since $\ker(\pi_s) = \langle z_s \rangle$. 

Thus the images $\theta(x_i)$ are very restricted, and we are nearly done. We apply $\theta$ to the braid relation $x_i x_{i+1} x_i = x_{i+1} x_i x_{i+1}$ for some $1 \leq i \leq n-2$. Using \eqref{Eq:theta_xi} and the centrality of $z_s$, we get $y_i^\varepsilon y_{i+1}^\varepsilon y_i^\varepsilon z_s^{2m_i + m_{i+1}} = y_{i+1}^\varepsilon y_{i}^\varepsilon y_{i+1}^\varepsilon z_s^{m_i + 2m_{i+1}}$. Since the braid relation also holds for the $y_i$ in $G_{n,s}$, we cancel and obtain $z_s^{2m_i + m_{i+1}} = z_s^{m_i + 2m_{i+1}}$, and since $z_s$ has infinite order thus $2m_i + m_{i+1} = m_i + 2m_{i+1}$, i.e.\ $m_i = m_{i+1}$. Since $i$ was arbitrary, we get $m_1 = \cdots = m_{n-1} = m$ for some fixed $m \in \Z$. Next, we apply $\theta$ to the relation $(x_1 \cdots x_{n-1})^n = z^r$. Since $z_s$ is central, \eqref{Eq:theta_xi} gives $\theta(x_1 \cdots x_{n-1}) = (y_1^\varepsilon \cdots y_{n-1}^\varepsilon) z_s^{(n-1)m}$, and hence, recalling $N=n(n-1)$, we get
\begin{equation}\label{Eq:image-of-Delta}
\theta( (x_1 \cdots x_{n-1})^n ) = (y_1^\varepsilon \cdots y_{n-1}^\varepsilon)^n z_s^{Nm}.
\end{equation}
We claim that $(y_1^\varepsilon \cdots y_{n-1}^\varepsilon)^n = z_s^{\varepsilon s}$. But this is easy: if $\varepsilon = 1$ then this is explicitly one of the defining relations of $G_{n,s}$, and if $\varepsilon = -1$, it follows from the fact that $(y_1^{-1} \cdots y_{n-1}^{-1})^n = z_s^{-s}$. This latter fact follows directly by using the mirror automorphism $\mu \colon \B_n \to \B_n$ mapping $\mu(\sigma_i) = \sigma_i^{-1}$ and that this maps the central element $\Delta^2 = (\sigma_1 \cdots \sigma_{n-1})^n$ to its inverse (since it reverses the exponent sum), so that $\Delta^{-2} = \mu(\Delta^2) = (\sigma_1^{-1} \cdots \sigma_{n-1}^{-1})^n$. Thus we have $(y_1^\varepsilon \cdots y_{n-1}^\varepsilon)^n = z_s^{\varepsilon s}$ for both cases of $\varepsilon = \pm 1$. Combining this with \eqref{Eq:image-of-Delta} and the earlier observation that $\theta(z_r) = z_s^k$ for some $k = \pm 1$, we get 
\[
z_s^{kr} = \theta(z_r^r) = \theta( (x_1 \cdots x_{n-1})^n) = z_s^{\varepsilon s} z_s^{Nm}
\]
and hence $kr = \varepsilon s + N m$, so that $kr \equiv \varepsilon s \pmod{N}$. Since $k\varepsilon = \pm 1$, we thus get that $r \equiv \pm s \pmod{N}$, which is what was to be shown. Thus we have proved Lemma~\ref{Lem:not-isomorphic}.

\

Next, we prove Lemma~\ref{Lem:same-quotients}. We will do this by proving that $G_{n,r} \times \Z \cong G_{n,s} \times \Z$, and then applying Baumslag's result from the introduction (Proposition~\ref{Prop:Baumslag}). The proof of this isomorphism is similar to the proof of \cite[Lemma~2]{Baumslag1974}, and mainly consists of keeping track of elementary number-theoretic information. 

We retain our earlier notation of the generators of $G_{n,r}$ resp.\ $G_{n,s}$ as $x_1, \dots, x_{n-1}, z_r$ resp. $y_1, \dots, y_{n-1}, z_s$. Let $u$ resp.\ $v$ denote the cyclic generator of the direct $\Z$-factor in $G_{n,r} \times \Z$ resp.\ $G_{n,s} \times \Z$. First, choose $a \in \Z$ such that $ar \equiv s \pmod{N}$, which is possible since both are units. Since then also $\gcd(a, N) = 1$, we can choose $c, d \in \Z$ such that $ad - Nc = 1$. Then $ds \equiv dar \equiv r \pmod{N}$, so if we set $m := \frac{ds-r}{N}$, then $m\in \Z$. We define a map 
\begin{align*}
\Phi \colon G_{n,s} \times \Z &\to G_{n,r} \times \Z \\ 
y_i &\mapsto x_i z_r^m u^{-cs}, \\
z_s &\mapsto z_r^d u^{-cN},\\
v &\mapsto z_r^{-1} u^a.
\end{align*}
We prove this is a homomorphism. Of course, all braid relations of $G_{n,s} \times \Z$ are preserved, since every $y_i$ is sent to $x_i$ multiplied by the same central element $z_r^m u^{-cs}$. Likewise, all centrality relations are preserved; that is, $\Phi(z_s)$ and $\Phi(v)$ are both central in $G_{n,r} \times \Z$. Thus we must only check that the relation $(y_1 \cdots y_{n-1})^n = z_s^s$ is preserved. We apply the map and find 
\begin{align*}
\Phi((y_1 \cdots y_{n-1})^n) = (x_1 z_r^m u^{-cs} \cdots x_{n-1} z_r^m u^{-cs})^n &= (x_1 \cdots x_{n-1})^n (z_r^m u^{-cs})^{n(n-1)} \\
&= z_r^r \cdot (z_r^m u^{-cs})^{N} = z_r^{r + mN} u^{-csN}
\end{align*}
using centrality of $z_r$ and $u$, and $N = n(n-1)$. But $mN = ds - r$ by definition of $m$, so $z_r^{r+mN} u^{-csN} = z_r^{ds} u^{-csN} = (z_r^d u^{-cN})^s$, which is precisely $\Phi(z_s^s)$. Thus $\Phi$ is a homomorphism. 

To see that $\Phi$ is an isomorphism, we could either note that all the above steps are reversible with some elementary number theory, giving an explicit map $G_{n,r} \times \Z \to G_{n,s} \times \Z$, but it is quicker to proceed as follows: first, $\Phi$ induces the identity on the quotient $Q_n = (G_{n,s} \times \Z) / \langle z_s, v \rangle$. On the other hand, $\Phi$ induces a map $\langle z_s, v \rangle \to \langle z_r, u \rangle$ on the central kernels, both of which are isomorphic to $\Z^2$. The map is represented, with respect to these ordered bases, by the matrix 
\[
\begin{pmatrix}
d & -1 \\ -cN & a
\end{pmatrix} \in \SL_2(\Z), \quad \text{since $ad - Nc = 1$,}
\]
and this being invertible, it is an automorphism of $\Z^2$. Hence $\Phi$ is an isomorphism, since it is an isomorphism on the central kernel and induces the identity on the quotient. Thus $G_{n,r} \times \Z \cong G_{n,s} \times \Z$, and we have proved Lemma~\ref{Lem:same-quotients} in view of Proposition~\ref{Prop:Baumslag}. 

The only final observation to be made is that each $G_{n,r}$ is indeed residually finite. But this is easy: by the proof of Lemma~\ref{Lem:same-quotients}, we have $G_{n,r} \times \Z \cong G_{n,1} \times \Z \cong \B_n \times \Z$. Magnus \cite{Magnus1969} proved that $\B_n$ is residually finite, and hence so too is $\B_n \times \Z \cong G_{n,r} \times \Z$. Since residual finiteness is preserved by taking subgroups, this proves that $G_{n,r}$ is residually finite. Thus, putting Lemma~\ref{Lem:not-isomorphic} and Lemma~\ref{Lem:same-quotients} together now completes the proof of our Theorem. \qed

\clearpage

\section{}

\noindent Concretely, if we take $n=4$, we get $N = n(n-1) = 12$. Taking $r = 5 \in (\Z / 12\Z)^\times$, we get 
\[
G_{4,5} = \pres{}{a,b,c,z}{aba = bab, \: bcb = cbc, \: [a,c]=1, \: z \text{ central}, \: (abc)^4 = z^5}
\]
which thus has the same finite quotients as $\B_4 \cong G_{4,1}$, but it is not isomorphic to $\B_4$. A lower bound for the size of the genus $\mathcal{C}(\B_n)$ is by Lemma~\ref{Lem:not-isomorphic} precisely half the number of units in $\Z / N\Z$, and hence $|\mathcal{C}(\B_n)| \geq \frac{\varphi( n(n-1))}{2}>1$, since $n \geq 4$. Note that when $n=3$, there is only one $G_{n,r}$ up to isomorphism, and hence we cannot conclude anything about the profinite rigidity of $\B_3$. It would be very interesting to know whether $\mathcal{C}(\B_n)$ is infinite or not, as the constructions in this note can only produce finite genera. 

We remark that Baumslag's result has also been used by Hempel \cite{Hempel2014} to obtain families of non-isomorphic groups with the same profinite completion arising as the fundamental group of Seifert fibered spaces with zero rational Euler number. Finally, we note that it is very plausible that same technique can be applied to some other Artin groups, particularly the Artin groups of type $D_n$, in much the same manner as in this article. It would be interesting to know precisely which infinite Artin groups are profinitely rigid; in the right-angled case, the recent work of Corson, Hughes, M\"oller, and Varghese \cite{Corson2026} is particularly relevant. 

\section*{Acknowledgements} 
\noindent I wish to thank Martin Bridson, Sam Hughes, and Alan Reid for their helpful comments and pointers to the literature. My initial proof of Lemma~\ref{Lem:not-isomorphic} was simplified by an observation made by GPT-5.6 Sol accessed via ChatGPT Pro. Every word in the article was written by its (human) author.

\bibliographystyle{amsalpha}
\bibliography{BraidGroupsNotPR.bib}

@article{Andrew2026,
    author = {Andrew, Naomi and Hensel, Sebastian and Hughes, Sam and Wade, Richard D.},
    title = {Problems on handlebody groups},
    journal = {Royal Society Open Science},
    volume = {13},
    number = {2},
    pages = {250502},
    year = {2026},
    issn = {2054-5703},
    doi = {10.1098/rsos.250502},
    url = {https://doi.org/10.1098/rsos.250502},
}

@article {Baumslag1974,
    AUTHOR = {Baumslag, Gilbert},
     TITLE = {Residually finite groups with the same finite images},
   JOURNAL = {Compositio Math.},
  FJOURNAL = {Compositio Mathematica},
    VOLUME = {29},
      YEAR = {1974},
     PAGES = {249--252},
      ISSN = {0010-437X,1570-5846},
   MRCLASS = {20E25},
  MRNUMBER = {357615},
MRREVIEWER = {A.\ H.\ Rhemtulla},
}

@article {Bell2006,
    AUTHOR = {Bell, Robert W. and Margalit, Dan},
     TITLE = {Braid groups and the co-{H}opfian property},
   JOURNAL = {J. Algebra},
  FJOURNAL = {Journal of Algebra},
    VOLUME = {303},
      YEAR = {2006},
    NUMBER = {1},
     PAGES = {275--294},
      ISSN = {0021-8693,1090-266X},
   MRCLASS = {20F36 (20F38 57M07)},
  MRNUMBER = {2253663},
MRREVIEWER = {Luisa\ Paoluzzi},
       DOI = {10.1016/j.jalgebra.2005.10.038},
       URL = {https://doi.org/10.1016/j.jalgebra.2005.10.038},
}

@article {Bridson2017,
    AUTHOR = {Bridson, Martin R. and Reid, Alan W. and Wilton, Henry},
     TITLE = {Profinite rigidity and surface bundles over the circle},
   JOURNAL = {Bull. Lond. Math. Soc.},
  FJOURNAL = {Bulletin of the London Mathematical Society},
    VOLUME = {49},
      YEAR = {2017},
    NUMBER = {5},
     PAGES = {831--841},
      ISSN = {0024-6093,1469-2120},
   MRCLASS = {57M27 (20E18 20E26)},
  MRNUMBER = {3742450},
MRREVIEWER = {Steffen\ Kionke},
       DOI = {10.1112/blms.12076},
       URL = {https://doi.org/10.1112/blms.12076},
}

@article {BMRS,
    AUTHOR = {Bridson, M. R. and McReynolds, D. B. and Reid, A. W. and
              Spitler, R.},
     TITLE = {Absolute profinite rigidity and hyperbolic geometry},
   JOURNAL = {Ann. of Math. (2)},
  FJOURNAL = {Annals of Mathematics. Second Series},
    VOLUME = {192},
      YEAR = {2020},
    NUMBER = {3},
     PAGES = {679--719},
      ISSN = {0003-486X,1939-8980},
   MRCLASS = {57M50 (11F06 20E18 20H10)},
  MRNUMBER = {4172619},
       DOI = {10.4007/annals.2020.192.3.1},
       URL = {https://doi.org/10.4007/annals.2020.192.3.1},
}

@article {BMRS21,
    AUTHOR = {Bridson, Martin R. and McReynolds, D. B. and Reid, Alan W. and
              Spitler, Ryan},
     TITLE = {On the profinite rigidity of triangle groups},
   JOURNAL = {Bull. Lond. Math. Soc.},
  FJOURNAL = {Bulletin of the London Mathematical Society},
    VOLUME = {53},
      YEAR = {2021},
    NUMBER = {6},
     PAGES = {1849--1862},
      ISSN = {0024-6093,1469-2120},
   MRCLASS = {20H10 (11F06)},
  MRNUMBER = {4386043},
MRREVIEWER = {Alexander\ W.\ Mason},
       DOI = {10.1112/blms.12546},
       URL = {https://doi.org/10.1112/blms.12546},
}

@article {Corson2026,
    AUTHOR = {Corson, Samuel M. and Hughes, Sam and M\"{o}ller, Philip and
              Varghese, Olga},
     TITLE = {Higman-{T}hompson groups and profinite properties of
              right-angled {C}oxeter groups},
   JOURNAL = {Selecta Math. (N.S.)},
  FJOURNAL = {Selecta Mathematica. New Series},
    VOLUME = {32},
      YEAR = {2026},
    NUMBER = {2},
     PAGES = {Paper No. 28, 23},
      ISSN = {1022-1824,1420-9020},
   MRCLASS = {20F55 (20E18 20E36 20E45 20F65)},
  MRNUMBER = {5043176},
       DOI = {10.1007/s00029-026-01130-4},
       URL = {https://doi.org/10.1007/s00029-026-01130-4},
}

@article {Dixon1982,
    AUTHOR = {Dixon, John D. and Formanek, Edward W. and Poland, John C. and
              Ribes, Luis},
     TITLE = {Profinite completions and isomorphic finite quotients},
   JOURNAL = {J. Pure Appl. Algebra},
  FJOURNAL = {Journal of Pure and Applied Algebra},
    VOLUME = {23},
      YEAR = {1982},
    NUMBER = {3},
     PAGES = {227--231},
      ISSN = {0022-4049,1873-1376},
   MRCLASS = {20E18},
  MRNUMBER = {644274},
MRREVIEWER = {S.\ P.\ Demushkin},
       DOI = {10.1016/0022-4049(82)90098-6},
       URL = {https://doi.org/10.1016/0022-4049(82)90098-6},
}

@article {Hempel2014,
	AUTHOR = {Hempel, John},
	TITLE = {Some $3$-manifold groups with the same finite quotients},
	JOURNAL = {arXiv preprint},
	FJOURNAL = {arXiv preprint},
	YEAR = {2014},
	NOTE = {available on arXiv:1409.3509},
}

@article {Hughes2025,
    AUTHOR = {Hughes, Sam and Kudlinska, Monika},
     TITLE = {On profinite rigidity amongst free-by-cyclic groups {I}: {T}he
              generic case},
   JOURNAL = {Proc. Lond. Math. Soc. (3)},
  FJOURNAL = {Proceedings of the London Mathematical Society. Third Series},
    VOLUME = {130},
      YEAR = {2025},
    NUMBER = {6},
     PAGES = {Paper No. e70059, 43},
      ISSN = {0024-6115,1460-244X},
   MRCLASS = {20E18 (20E26 20E36 20F65 20F67 20J05 20J06)},
  MRNUMBER = {4920397},
MRREVIEWER = {Benjamin\ Klopsch},
       DOI = {10.1112/plms.70059},
       URL = {https://doi.org/10.1112/plms.70059},
}

@article {Kolay2023,
    AUTHOR = {Kolay, Sudipta},
     TITLE = {Smallest noncyclic quotients of braid and mapping class
              groups},
   JOURNAL = {Geom. Topol.},
  FJOURNAL = {Geometry \& Topology},
    VOLUME = {27},
      YEAR = {2023},
    NUMBER = {6},
     PAGES = {2479--2496},
      ISSN = {1465-3060,1364-0380},
   MRCLASS = {20F36 (20F65 57K20)},
  MRNUMBER = {4634752},
MRREVIEWER = {Clement\ Radu\ Popescu},
       DOI = {10.2140/gt.2023.27.2479},
       URL = {https://doi.org/10.2140/gt.2023.27.2479},
}

@article {Magnus1969,
    AUTHOR = {Magnus, W.},
     TITLE = {Residually finite groups},
   JOURNAL = {Bull. Amer. Math. Soc.},
  FJOURNAL = {Bulletin of the American Mathematical Society},
    VOLUME = {75},
      YEAR = {1969},
     PAGES = {305--316},
      ISSN = {0002-9904},
   MRCLASS = {20.27},
  MRNUMBER = {241525},
MRREVIEWER = {R.\ H.\ Fox},
       DOI = {10.1090/S0002-9904-1969-12149-X},
       URL = {https://doi.org/10.1090/S0002-9904-1969-12149-X},
}

@article {Nikolov2007,
    AUTHOR = {Nikolov, Nikolay and Segal, Dan},
     TITLE = {On finitely generated profinite groups. {I}. {S}trong
              completeness and uniform bounds},
   JOURNAL = {Ann. of Math. (2)},
  FJOURNAL = {Annals of Mathematics. Second Series},
    VOLUME = {165},
      YEAR = {2007},
    NUMBER = {1},
     PAGES = {171--238},
      ISSN = {0003-486X,1939-8980},
   MRCLASS = {20E18 (20E32 20F12)},
  MRNUMBER = {2276769},
MRREVIEWER = {Benjamin\ Klopsch},
       DOI = {10.4007/annals.2007.165.171},
       URL = {https://doi.org/10.4007/annals.2007.165.171},
}

@article {Scherich2023,
    AUTHOR = {Scherich, Nancy and Verberne, Yvon},
     TITLE = {Finite image homomorphisms of the braid group and its
              generalizations},
   JOURNAL = {Glasg. Math. J.},
  FJOURNAL = {Glasgow Mathematical Journal},
    VOLUME = {65},
      YEAR = {2023},
    NUMBER = {2},
     PAGES = {430--445},
      ISSN = {0017-0895,1469-509X},
   MRCLASS = {20F36 (57K12)},
  MRNUMBER = {4625994},
MRREVIEWER = {Nick\ Salter},
       DOI = {10.1017/S0017089523000022},
       URL = {https://doi.org/10.1017/S0017089523000022},
}

\end{document}